\documentclass[12pt, a4paper]{article}

\usepackage{amsmath}
\usepackage{amsthm}
\usepackage{amssymb}
\usepackage{color}

\newtheorem{theorem}{Theorem}
\newtheorem{proposition}[theorem]{Proposition}
\newtheorem{lemma}[theorem]{Lemma}

\newtheorem{conjecture}[theorem]{Conjecture}

\title{A converse to the neo-classical inequality with an application to the Mittag-Leffler function}
\author{Stefan Gerhold \\
TU Wien \\
\tt{sgerhold@fam.tuwien.ac.at}
\and
Thomas Simon\\
Universit{\'e} de Lille\\
\tt{thomas.simon@univ-lille.fr}
}

\date{\today}

\numberwithin{equation}{section}
\numberwithin{theorem}{section}

\begin{document}

\maketitle

\begin{abstract}
We prove two inequalities for the Mittag-Leffler function, namely
that the function $\log E_\alpha(x^\alpha)$ is sub-additive for
$0<\alpha<1,$ and super-additive for $\alpha>1.$ These assertions
follow from two new binomial inequalities, one of which is a converse
to the neo-classical inequality. The proofs use a generalization
of the binomial theorem due to Hara and Hino (Bull.\ London Math.\ Soc. 2010).
For $0<\alpha<2,$ we also show that $E_\alpha(x^\alpha)$ is log-concave
resp.\ log-convex, using analytic as well as probabilistic
arguments.
\end{abstract}


\section{Introduction and main results}

The Mittag-Leffler function
\[
 E_\alpha(x) := \sum_{k=0}^\infty \frac{x^k}{\Gamma(\alpha k + 1)},
 \quad \alpha>0,
\]
is a well-known special function with a large number of applications in pure
and applied mathematics; see~\cite{GoKiMaRo20,HaMaSa11} for surveys.
Clearly, we have $E_1(x)=e^x.$
Somewhat surprisingly, the ``identity''
\begin{equation}\label{eq:wrong}
  E_\alpha((x+y)^\alpha) \stackrel{?}{=} E_\alpha(x^\alpha) E_\alpha(y^\alpha)
\end{equation}
has been used in a few papers. As discussed in~\cite{El16,PeLi10}, it is not
correct for $\alpha\neq1$. In \cite{PeLi10}, a correct identity involving integrals of $E_\alpha(x^\alpha)$ is
proven, which reduces to $e^{x+y}=e^x e^y$ as $\alpha\to1$.
Besides this, it seems natural to ask whether the left and right hand sides
of~\eqref{eq:wrong} are comparable. This is indeed the case:
\begin{theorem}\label{thm:ml}
 For $0<\alpha<1,$ we have
 \begin{equation}\label{eq:ML1}
   E_\alpha((x+y)^\alpha) \leq E_\alpha(x^\alpha) E_\alpha(y^\alpha),\quad
   x,y\geq 0,
 \end{equation}
 and for $\alpha>1$
 \begin{equation}\label{eq:ML2}
   E_\alpha((x+y)^\alpha) \geq E_\alpha(x^\alpha) E_\alpha(y^\alpha),\quad
   x,y\geq 0.
 \end{equation}
 These inequalities are strict for $x,y>0$.
\end{theorem}
The strictness assertion strengthens the observation made at the beginning
of Section~2 of~\cite{El16}, where it is argued that the validity
of~\eqref{eq:wrong} \emph{for all} $x,y\geq0$ implies
$\alpha=1.$ According to Theorem~\ref{thm:ml}, this equality \emph{never} holds,
except in the obvious cases ($\alpha=1,$ or $xy=0$).
Although apparently not made explicit in the literature, the 
 lower estimate
\[
  E_\alpha((x+y)^\alpha) \geq \alpha E_\alpha(x^\alpha) E_\alpha(y^\alpha),\quad
   x,y\geq 0,\ 0<\alpha<1,
\]
complementing~\eqref{eq:ML1},  follows from the calculation
\begin{align}\label{eq:nc to ml}
\begin{split}
   E_\alpha((x+y)^\alpha) &=
     \sum_{k=0}^\infty \frac{(x+y)^{\alpha k}}{\Gamma(\alpha k+1)} \\
   &\geq \sum_{k=0}^\infty \frac{\alpha}{\Gamma(\alpha k+1)}
     \sum_{j=0}^k \binom{\alpha k}{\alpha j} x^{\alpha j}y^{\alpha(k-j)}\\
   &= \alpha \sum_{k=0}^\infty \sum_{j=0}^k \frac{x^{\alpha j}y^{\alpha(k-j)}}
     {\Gamma(\alpha j+1)\Gamma(\alpha(k-j)+1)} = \alpha E_\alpha(x^\alpha)
        E_\alpha(y^\alpha).
\end{split}
\end{align}
In the second line, we have used the following result.
\begin{theorem}[Neo-classical inequality; Theorem~1.2 in~\cite{HaHi10}]\label{thm:nc}
  For $k\in\mathbb N$ and $0<\alpha<1,$ we have
  \begin{equation}\label{eq:nc}
     \alpha \sum_{j=0}^k \binom{\alpha k}{\alpha j} x^{\alpha j} y^{\alpha(k-j)} \leq (x+y)^{\alpha k},
     \quad x,y\geq0.
  \end{equation}
\end{theorem}
With a slightly weaker factor of~$\alpha^2$ instead of~$\alpha$, this result
was proven by Lyons in 1998, who also coined the term neo-classical inequality,
in a pioneering paper on rough path theory~\cite{Ly98}. Later, it has been applied by several other authors,
see e.g.~\cite{BoGeSo20,CrLiLy15,FrRi14,In12}.
Analogously to~\eqref{eq:nc to ml}, it is clear that Theorem~\ref{thm:ml} follows
from the following new binomial inequalities.
\begin{theorem}\label{thm:cnc}
  For $k\in\mathbb N$ and $0<\alpha<1,$ we have the
  converse neo-classical inequality
  \begin{equation}\label{eq:cnc1 xy}
     \sum_{j=0}^k \binom{\alpha k}{\alpha j} x^{\alpha j} y^{\alpha(k-j)} \geq (x+y)^{\alpha k},
     \quad x,y\geq0,
  \end{equation}
  and for $\alpha>1$ we have
  \begin{equation}\label{eq:cnc2 xy}
     \sum_{j=0}^k \binom{\alpha k}{\alpha j} x^{\alpha j} y^{\alpha(k-j)} \leq (x+y)^{\alpha k},
     \quad x,y\geq0.
  \end{equation}
  These inequalities are strict for $x,y>0$.
\end{theorem}

The inequalities \eqref{eq:nc}--\eqref{eq:cnc2 xy} look deceptively simple.
For instance, it is not obvious that
the proof of~\eqref{eq:cnc2 xy} -- at least with the approach used here --
is much harder when the integer $\lfloor \alpha \rfloor$ is even
than when it is odd.
Lyons's proof of (the weaker version of)~\eqref{eq:nc} applies the maximum principle for sub-parabolic
functions in a non-trivial way. Hara and Hino~\cite{HaHi10}
use fractional calculus to derive an extension of the binomial theorem (see
Theorem~\ref{thm:binom} below), which immediately implies Theorem~\ref{thm:nc}.
This extended binomal theorem will be our starting point when proving
Theorem~\ref{thm:cnc}.
We also note that in a preliminary version of~\cite{FrRi13} (available at arXiv:1104.0577v2 [math.PR]),
it was shown that a multinomial extension can be derived
from~\eqref{eq:nc}, by induction over the number of variables.

Our first theorem, Theorem~\ref{thm:ml}, says that the function
$\log E_\alpha(x^\alpha)$ is sub- resp.\ super-additive on
$\mathbb{R}^+=(0,\infty).$
In Section~\ref{se:ml proofs}, we will show stronger statements
for $0<\alpha<2$: The function $E_\alpha(x^\alpha)$ is log-concave
for $0<\alpha<1,$ and log-convex for $1<\alpha<2.$
Sections~\ref{se:prelim}--\ref{se:even}
are devoted to proving Theorem~\ref{thm:cnc}.
Some preliminaries and the plan of the proof are given
in Section~\ref{se:prelim}. In that section we also state a conjecture concerning
a converse inequality to~\eqref{eq:cnc2 xy}.

\section{Proof of Theorem~\ref{thm:ml} for $0<\alpha<2$ and
related statements}\label{se:ml proofs}

For brevity, we do not discuss strictness in this section. This would be straightforward,
and the strictness assertion in Theorem~\ref{thm:ml} will follow
anyways from Theorem~\ref{thm:cnc}.
The following easy fact is well-known; see, e.g.,~\cite{Br64}.
\begin{lemma}
  Let $f:[0,\infty) \to \mathbb R$ be convex with $f(0)=0.$ Then~$f$
  is superadditive, i.e.,
  \[
    f(x+y)\geq f(x)+f(y),\quad x,y\geq0.
  \]
\end{lemma}
Thus, the following theorem implies~\eqref{eq:ML1},
and~\eqref{eq:ML2} for $1<\alpha<2$.
\begin{theorem}\label{thm:logcc}
  For $0<\alpha<1,$ the function $x\mapsto E_\alpha(x^{\alpha})$ is
  log-concave on $\mathbb{R}^+$.
  For $1<\alpha<2,$ it is log-convex.
\end{theorem}
For $\alpha>2,$ it seems that $E_\alpha(x^\alpha)$ is \emph{not} log-convex.
For instance, $E_4(x^4)=\tfrac12(\cos x+\cosh x),$ and
\[
  \big(\log E_4(x^4) \big)'' =
  \frac{2 \sin x\ \sinh x}{(\cos x+\cosh x)^2}
\]
changes sign. To prove~\eqref{eq:ML2} for $\alpha>2,$ we thus
rely on the binomial inequality~\eqref{eq:cnc2 xy}, which is proven later.
\begin{proof}[Proof of Theorem~\ref{thm:logcc}]
  We start from the representation, given in~(3.4) of~\cite{TS15},
  \[
    E_\alpha(x^{\alpha}) = \frac{e^x}{\alpha}- \varphi_\alpha (x),\quad
    x>0,\ 0<\alpha<1,
  \]
  where
  \[
    \varphi_\alpha(x):= \frac{\sin \alpha \pi}{\pi}
    \int_{0}^\infty \frac{t^{\alpha-1}e^{-xt}}
      {t^{2\alpha} -2 \cos(\alpha \pi)t^\alpha +1}\, dt
  \]
  is a completely monotone function. (Recall that a smooth function~$f$ on $\mathbb{R}^+$ is completely
  monotone if
  \[
    (-1)^n f^{(n)}(x)\geq 0,\quad n\geq0,\ x>0.)
  \]
  By well-known closure properties of completely monotone functions
  (see e.g. Corollaries~1.6 and~1.7 in~\cite{ScSoVo12}),
  \[
    \frac{1}{E_\alpha(x^{\alpha})} = \alpha e^{-x}
      \sum_{n=0}^\infty \big(\alpha e^{-x} \varphi_\alpha(x) \big)^n
  \]
  is completely monotone as well and hence log-convex. Thus, $E_\alpha(x^{\alpha})$ is log-concave.
  
  Now suppose that $1<\alpha<2.$  Setting $\beta = \alpha/2\in (1/2,1),$ we have  
\[
E_\alpha(x^\alpha)  =  \frac{E_\beta(x^\beta) + E_\beta (-x^\beta)}{2} =  \frac{e^x}{\alpha} + \frac{E_\beta(-x^\beta) - \varphi_\beta(x)}{2}\cdot\]
Inserting the well-known representation
$$E_\beta(-x^\beta) = \frac{\sin \beta \pi}{\pi}
    \int_{0}^\infty \frac{t^{\beta-1}e^{-xt}}
      {t^{2\beta} +2 \cos(\beta \pi)t^\beta +1}dt,$$
which is e.g.\ a consequence of the Perron-Stieltjes inversion formula applied to (3.7.7) in \cite{GoKiMaRo20}, and making some trigonometric simplifications, we get the representation
\[
E_\alpha(x^\alpha)  =  \frac{e^x}{\alpha} - 
\frac{\sin \alpha \pi}{\pi}
    \int_{0}^\infty \frac{t^{\alpha-1}e^{-xt}}
      {t^{2\alpha} - 2 \cos(\alpha \pi)t^\alpha +1}\,dt,
      \]
which is also given in (3.6) of \cite{TS15}. This implies     
\[
  \log(E_\alpha(x^{\alpha})) = x - \log\alpha + \log(\psi_\alpha(x))
  \]
with 
\[\psi_\alpha(x) := 1 -\frac{\alpha \sin \alpha \pi}{\pi}
    \int_{0}^\infty \frac{t^{\alpha-1}e^{-x(1+t)}}
      {t^{2\alpha} - 2 \cos(\alpha \pi)t^\alpha +1}\,dt\]
a completely monotone function. Thus, $E_\alpha(x^{\alpha})$ is log-convex.
 \end{proof}
Observe that Theorem \ref{thm:logcc} extends to the boundary case $\alpha=2,$ where
\[ E_2(x^2) = \cosh x\] 
is clearly log-convex. Let us also mention an alternative probabilistic argument for \eqref{eq:ML1} based on the $\alpha$-stable subordinator $\{ Z_t^{(\alpha)}, \; t\geq0\}$ with normalization $\mathbb{E} \big[ e^{-\lambda Z_t^{(\alpha)}}\big] = e^{-t\lambda^{\alpha}}.$ It is indeed well-known -- see e.g.\ Exercise~50.7 in \cite{Sato99} -- that  
$$E_\alpha(x^\alpha)=\mathbb{E}[e^{R^{(\alpha)}_x}],$$ 
where $R^{(\alpha)}_x := \inf\{t > 0 : Z^{(\alpha)}_t > x\}.$ Now if ${\tilde R}^{(\alpha)}_y$ is an independent copy of $R^{(\alpha)}_y,$ the Markov property implies
\begin{align*}
  R^{(\alpha)}_x + {\tilde R}^{(\alpha)}_y
  &\stackrel{\mathrm{d}}{=}\inf\big\{t > \! R^{(\alpha)}_x : Z_t^{(\alpha)} > Z_{R^{(\alpha)}_x}^{(\alpha)} +y\big\} \\ &\ge \inf\big\{t > \! R^{(\alpha)}_x : Z_t^{(\alpha)}\! > x +y\big\} \\ &= \inf\big\{t > 0 : Z_t^{(\alpha)}\! > x +y\big\} \stackrel{\mathrm{d}}{=} R^{(\alpha)}_{x+y},
\end{align*}
where the inequality follows from the obvious fact that $Z_{R^{(\alpha)}_x}^{(\alpha)} \ge x.$ 
This shows the desired inequality
$$E_\alpha(x^\alpha)E_\alpha(x^\alpha)\ge E_\alpha((x+y)^\alpha).$$
To conclude this section, we give some related results for the function
$E_\alpha(x).$
\begin{proposition}
   For $0<\alpha<1,$ the  function $x\mapsto E_\alpha(x)$ is
  log-convex on~$\mathbb{R}$. For $\alpha>1,$ it is log-concave on~$\mathbb{R}^+$.
\end{proposition}
\begin{proof}
The logarithmic derivative of $E_\alpha(x)$ is the ratio of series
$$\frac{E_\alpha'(x)}{E_\alpha(x)} = \frac{{\displaystyle \sum_{n\ge 0} \frac{x^n}{\Gamma(\alpha +\alpha n)}}}{{\displaystyle\sum_{n\ge 0} \frac{\alpha \, x^n}{\Gamma(1+\alpha n)}}},$$
and it is clear by log-convexity of the gamma function that the sequence
$$n \mapsto \frac{\Gamma(1+\alpha n)}{\Gamma (\alpha + \alpha n)}$$
is increasing for $\alpha < 1$ and decreasing for $\alpha > 1.$ By Biernacki and Krzy{\. z}'s lemma (see \cite{BP55}), this shows that
$$x\mapsto \frac{E_\alpha'(x)}{E_\alpha(x)}$$
is non-decreasing on $\mathbb{R}^+$ for $\alpha < 1$ and non-increasing on $\mathbb{R}^+$ for $\alpha > 1.$ Finally, the log-convexity of $E_\alpha(x)$ on the whole real line for $0 < \alpha < 1$ is a consequence of H\"older's inequality and the classic m.g.f. representation
$$E_\alpha(x)=\mathbb{E}[e^{x R^{(\alpha)}_1}],$$
where we have used the above notation and the easily established self-similar identity $ R^{(\alpha)}_x \stackrel{\mathrm{d}}{=} x R^{(\alpha)}_1.$
\end{proof}

For $\alpha\ge 2,$ there is actually a stronger result. It has been shown by Wiman~\cite{Wi05} that the zeros of the Mittag-Leffler function are real and negative for $\alpha \ge 2.$ As this function is of order $1/\alpha<1,$ the Hadamard factorization theorem (Theorem XI.3.4 in~\cite{Co78}) implies that
  \[
    \frac{1}{E_\alpha(x)} = \prod_{n=1}^\infty\Big(1+\frac{x}{x_{n,\alpha}}\Big)^{-1},
  \]
  where
  \[
    0 < x_{1,\alpha} \leq x_{2,\alpha} \leq \cdots
  \]
  are the absolute values of the zeros of $E_\alpha(x).$ We conclude that $1/E_\alpha(x)$ is completely monotone, and thus log-convex, on ${\mathbb R}^+.$ An interesting open question is whether $1/E_\alpha(x)$ remains completely monotone on ${\mathbb R}^+$ for $1 < \alpha < 2.$ Unfortunately, the above argument fails because the large zeroes of $E_\alpha(x)$ have non-trivial imaginary part (see Proposition 3.13 in \cite{GoKiMaRo20}). 

\section{Proof of Theorem~\ref{thm:cnc}: preliminaries}\label{se:prelim}

By symmetry and scaling, it is clearly sufficient to prove
\begin{equation}\label{eq:cnc1}
  \sum_{j=0}^k \binom{\alpha k}{\alpha j} \lambda^{\alpha j}>(1+\lambda)^{\alpha k},
  \quad \alpha \in (0,1),\ \lambda \in (0,1],\ k\in\mathbb N,
\end{equation}
and
\begin{equation}\label{eq:cnc2}
  \sum_{j=0}^k \binom{\alpha k}{\alpha j} \lambda^{\alpha j}<(1+\lambda)^{\alpha k},
  \quad \alpha >1,\ \lambda \in (0,1],\ k\in\mathbb N.
\end{equation}
As a sanity check, we verify these statements for $\lambda>0$ sufficiently small:
\[
  1 + \binom{\alpha k}{\alpha}\lambda^{\alpha} + \mathrm{O}(\lambda^{2\alpha})
  > 1 +\alpha k \lambda + \mathrm{O}(\lambda^2), \quad \alpha \in (0,1),\ k\in\mathbb N,
\]
and
\[
  1 +  \mathrm{O}(\lambda^{\alpha})
  < 1 +\alpha k \lambda + \mathrm{O}(\lambda^2), \quad \alpha >1,\ k\in\mathbb N.
\]
These inequalities are obviously correct for small $\lambda>0,$ the first
one by $\binom{\alpha k}{\alpha}>0.$
We now recall a remarkable generalization of the binomial
theorem, due to Hara and Hino~\cite{HaHi10} . Their proof, using fractional Taylor
series, builds on earlier work by Osler~\cite{Os71}.
Following~\cite{HaHi10}, for $\alpha>0$ define
\begin{align}
   K_\alpha &:= \{ \omega \in \mathbb C : \omega^\alpha =1 \} \notag \\
   &= \{ e^{i\theta} : -\pi<\theta\leq \pi,\ e^{i\theta \alpha}=1 \} \notag \\
   &=\Big\{ \exp\Big(\frac{2k\pi i}{\alpha}\Big) : k\in\mathbb{Z},\
    -\frac{\alpha}{2}< k \leq \frac{\alpha}{2} \Big\}. \label{eq:K exp}
\end{align}
For $t,\lambda\in(0,1]$ and $k\in\mathbb N,$ define
\begin{equation}\label{eq:def F}
  F(t,\lambda,k):=  t^{\alpha-1}(1-t)^{\alpha k}
  \bigg( \frac{1}{|t^\alpha-\lambda^\alpha e^{-i\alpha \pi}|^2}
  + \frac{\lambda^{\alpha k}}{|e^{-i\alpha \pi}-(\lambda t)^\alpha |^2}
  \bigg).
\end{equation}
\begin{theorem}[Theorem~3.2 in~\cite{HaHi10}]\label{thm:binom}
  Let $\alpha>0,$ $\lambda\in(0,1]$ and $k\in\mathbb{N}_0.$ Then we have
  \begin{equation}\label{eq:binom}
    \alpha \sum_{j=0}^k \binom{\alpha k}{\alpha j}\lambda^{\alpha j}
    =\sum_{\omega \in K_\alpha}(1+\lambda \omega)^{\alpha k}
     - \frac{\alpha \lambda^\alpha \sin \alpha \pi}{\pi}
    \int_0^1 F(t,\lambda,k) dt.
  \end{equation}
\end{theorem}
In~\cite{HaHi10}, the theorem was stated for $k\in\mathbb N$, but it is not
hard to see that the proof also works for $k=0.$
Clearly, the classical binomial theorem is recovered from~\eqref{eq:binom}
by putting $\alpha=1.$
As noted in~\cite{HaHi10}, for $\alpha\in(0,1)$ we have $K_\alpha=\{1\},$
and thus Theorem~\ref{thm:nc} is an immediate consequence of Theorem~\ref{thm:binom}.
Hara and Hino also mention that Theorem~\ref{thm:binom} implies
 \begin{equation*}
     \alpha \sum_{j=0}^k \binom{\alpha k}{\alpha j} x^{\alpha j} y^{\alpha(k-j)} \geq (x+y)^{\alpha k},
     \quad x,y\geq0,\  \alpha\in(1,2],
  \end{equation*}
  a partial converse to~\eqref{eq:cnc2 xy}. It seems that this inequality
  does not hold for $\alpha>2.$ 
   We leave it to future research to find
  an appropriate inequality comparing the binomial sum
  with $(x+y)^{\alpha k}$ for $\alpha>2.$ Different methods than in the subsequent
  sections will be required, as a \emph{lower} estimate
  for $\sum_{\omega \in K_\alpha}(1+\lambda \omega)^{\alpha k}$
  is needed. The following statement might be true:
   \begin{conjecture}
    For $k\in\mathbb N$ and $\alpha>2,$ we have
    \[
       2^{\alpha-1} \sum_{j=0}^k \binom{\alpha k}{\alpha j} x^{\alpha j}    
              y^{\alpha(k-j)} \geq (x+y)^{\alpha k},
     \quad x,y\geq0.
     \]
  \end{conjecture}
  The factor $2^{\alpha-1}=\sup_{0<\lambda\leq1}(1+\lambda)^\alpha/(1+\lambda^\alpha)$ would be sharp for $k=1,$ $x=\lambda\in(0,1],$ $y=1.$
    In the following sections,
we will use arguments based on Theorem~\ref{thm:binom} to prove~\eqref{eq:cnc1} and~\eqref{eq:cnc2},
which imply Theorem~\ref{thm:cnc}, from which Theorem~\ref{thm:ml} follows.
The proof of~\eqref{eq:cnc1} is presented in Section~\ref{se:a01}. It profits from
the fact that
\[
  \sum_{\omega \in K_\alpha}(1+\lambda \omega)^{\alpha k} = (1+\lambda)^{\alpha k},
  \quad \alpha \in (0,1),
\]
and requires only obvious properties of the function~$F$
defined in~\eqref{eq:def F}.
At the beginning of Section~\ref{se:odd}, we show that, for
$2\leq\alpha\in\mathbb N,$ the inequality~\eqref{eq:cnc2} immediately
follows from the classical binomial theorem. We then continue with the case
where $\alpha>1$ is not an integer, and $\lfloor \alpha \rfloor$ is odd.
Then, the set $K_\alpha$ has $\lfloor \alpha \rfloor$ elements,
and the crude estimate
\[
  \sum_{\omega \in K_\alpha}(1+\lambda \omega)^{\alpha k} \leq
  \lfloor \alpha \rfloor (1+\lambda)^{\alpha k}
\]
suffices to show~\eqref{eq:cnc2}, again using only simple properties of~$F.$
The case where $\lfloor \alpha \rfloor$ is even is more involved,
and is handled in Section~\ref{se:even}.
In this case, $|K_\alpha|=\lceil \alpha \rceil$, and the obvious estimate
\begin{equation}\label{eq:too weak}
  \sum_{\omega \in K_\alpha}(1+\lambda \omega)^{\alpha k} \leq
  \lceil \alpha \rceil (1+\lambda)^{\alpha k}
\end{equation}
is too weak to lead anywhere. We first show (Lemma~\ref{le:big la})
that, for $\lambda\in[\tfrac12,1],$ the right hand side of~\eqref{eq:too weak}
can be strengthened to $\alpha  (1+\lambda)^{\alpha k},$
and that~\eqref{eq:cnc2} easily follows from this for these values of~$\lambda.$
For smaller~$\lambda>0,$ more precise estimates for the sum
$\sum_{\omega \in K_\alpha}(1+\lambda \omega)^{\alpha k}$
and the integral in~\eqref{eq:binom} are needed, which are developed
in the remainder of Section~\ref{se:even}.

\section{Proof of Theorem~\ref{thm:cnc} for $0<\alpha<1$}\label{se:a01}

As mentioned above, it suffices to prove~\eqref{eq:cnc1}.
For $\alpha\in(0,1),$ we have $K_\alpha=\{1\}.$ Since $\sin \alpha \pi>0,$ we
see from~\eqref{eq:binom} that the desired inequality~\eqref{eq:cnc1} is equivalent to
\begin{equation}\label{eq:a1}
  \int_0^1 F(t,\lambda,k)dt< G(\lambda,k),\quad \lambda\in(0,1],\ k\in\mathbb N,
\end{equation}
where
\[
  G(\lambda,k) := \frac{\pi(1-\alpha)}{\alpha \lambda^\alpha \sin \alpha \pi}
    (1+\lambda)^{\alpha k}.
\]
Define
\[
  \tilde{\delta} :=  \inf_{k\in\mathbb N}\Big((1+\lambda)^{\alpha k}-1)\Big)^{\frac{1}{\alpha k}}
  =(1+\lambda) \inf_{k\in\mathbb N}\Big(1-(1+\lambda)^{-\alpha k}\Big)^{\frac{1}{\alpha k}}>0.
\]
This number is positive, as it is defined by the infimum of a sequence of positive numbers
converging to a positive limit.
Moreover, we define
\[
  \delta := \tfrac12 \min\{1,\tilde{\delta}\}>0.
\]
By definition, $F(t,\lambda,k)$ decreases w.r.t.~$k$, and so
\[
 F(t,\lambda,k) \leq F(t,\lambda,0),\quad t,\lambda \in(0,1],\ k \in \mathbb N.
\]
Applying Theorem~\ref{thm:binom} with $k=0$ yields
\begin{equation*}
  \int_0^1 F(t,\lambda,0)dt = G(\lambda,0), \quad \lambda\in(0,1].
\end{equation*}
By these two observations,
\begin{equation}\label{eq:delta1}
  \int_0^{1-\delta}F(t,\lambda,k)dt \leq \int_0^{1-\delta}F(t,\lambda,0)dt
  \leq G(\lambda,0).
\end{equation}
Since $\lambda^{\alpha k}\leq 1,$ and
$(1-t)^{\alpha k}$ decreases w.r.t.~$t$, it is clear from the definition of~$F$ that
\begin{equation}\label{eq:delta2}
  \int_{1-\delta}^1  F(t,\lambda,k)dt
  \leq \delta^{\alpha k} \int_0^1 F(t,\lambda,0)dt = \delta^{\alpha k}G(\lambda,0).
\end{equation}
By definition of~$\delta$, we have
\begin{equation}\label{eq:delta lambda}
  1+\delta^{\alpha k}<(1+\lambda)^{\alpha k}, \quad
  \lambda\in(0,1],\ k\in\mathbb N.
\end{equation}
Now note that
\begin{equation}\label{eq:delta lambda2}
  \int_0^{1}F(t,\lambda,k)dt \leq (1+\delta^{\alpha k})G(\lambda,0) < G(\lambda,k),
\end{equation}
where the first estimate follows from~\eqref{eq:delta1} and~\eqref{eq:delta2},
and the second one from~\eqref{eq:delta lambda}. Thus, \eqref{eq:a1} is established.

\section{Proof of Theorem~\ref{thm:cnc} for $\alpha\in\mathbb N$
or $\lfloor\alpha\rfloor$ odd}\label{se:odd}

If $2\leq \alpha\in\mathbb N$ is an integer, then the proof of~\eqref{eq:cnc2} is very easy,
as we are  dealing with a classical binomial sum with some summands removed:
\[
  \sum_{j=0}^k \binom{\alpha k}{\alpha j} \lambda^{\alpha j}
  < \sum_{j=0}^{\alpha k} \binom{\alpha k}{ j} \lambda^{j}=(1+\lambda)^{\alpha k}.
\]
In the remainder of this section, we prove~\eqref{eq:cnc2}
for $1<\alpha\notin \mathbb N$ with $\lfloor\alpha\rfloor$ odd.
Our approach is similar to the preceding section.
First, observe that
\begin{align*}
    |K_\alpha| = \Big \lfloor \frac{\alpha}{2} \Big \rfloor - 
        \Big \lfloor {-\frac{\alpha}{2}} \Big \rfloor 
       = \lfloor \alpha \rfloor,
\end{align*}
and that
\begin{align*}
  \alpha(1+\lambda)^{\alpha k}&- \sum_{\omega \in K_\alpha}(1+\lambda \omega)^{\alpha k}\\
  &\geq  \alpha(1+\lambda)^{\alpha k}- \lfloor \alpha \rfloor (1+\lambda)^{\alpha k}
  > \alpha -  \lfloor \alpha \rfloor.
\end{align*}
Therefore, the sequence
\[
  A_k:=\bigg( \frac{\alpha(1+\lambda)^{\alpha k} 
            - \sum_{\omega \in K_\alpha}(1+\lambda \omega)^{\alpha k}}
            {\alpha -  \lfloor \alpha \rfloor}
        -1 \bigg)^\frac{1}{\alpha k}, \quad k\in\mathbb N,
\]
is well-defined and positive, and
\begin{align*}
  \lim_{k\to\infty} A_k &=
  (1+\lambda) \lim_{k\to\infty}(\alpha-1)^{\frac{1}{\alpha k}} \\
  &\qquad\qquad\times \bigg(
     \frac{1
            - \sum_{\omega \in K_\alpha\setminus\{1\}}
               \frac{(1+\lambda \omega)^{\alpha k}}{(\alpha-1)(1+\lambda)^{\alpha k}}}
            {\alpha -  \lfloor \alpha \rfloor}-\frac{1}{(\alpha-1)(1+\lambda)^{\alpha k}}
  \bigg)^\frac{1}{\alpha k}\\
   &=1+\lambda>0.
\end{align*}
We can thus define the positive number
\[
  \hat{\delta} := \frac12\min\Big\{1, \inf_{k\in\mathbb N} A_k\Big\}>0.
\]
Since odd $\lfloor \alpha \rfloor$ implies $\sin \alpha \pi <0$
for $\alpha\notin\mathbb N,$ it is clear from~\eqref{eq:binom}
that~\eqref{eq:cnc2} is equivalent to
\begin{equation}\label{goal}
  \int_0^1 F(t,\lambda,k)dt < \hat{G}(\lambda,k),\quad \lambda\in(0,1],\ k\in\mathbb N,
\end{equation}
where
\[
  \hat{G}(\lambda,k):=
  -\frac{\pi}{\alpha \lambda^\alpha \sin \alpha \pi}\Big(
    \alpha(1+\lambda)^{\alpha k}- \sum_{\omega \in K_\alpha}(1+\lambda \omega)^{\alpha k}
  \Big).
\]
By the same argument that gave us the first inequality
in~\eqref{eq:delta lambda2}, we find
\[
  \int_0^1 F(t,\lambda,k)dt \leq (1+\hat{\delta}^{\alpha k})\hat{G}(\lambda,0),
  \qquad \lambda\in(0,1],\ k\in\mathbb N.
\]
Note that
\[
  \hat{G}(\lambda,0)=
  -\frac{\pi}{\alpha \lambda^\alpha \sin \alpha \pi}(
    \alpha- \lfloor \alpha \rfloor
  ).
\]
The proof of~\eqref{eq:cnc2} with $\lfloor \alpha \rfloor$ odd will thus be finished
if we can show that
\[
   (1+\hat{\delta}^{\alpha k})\hat{G}(\lambda,0) < \hat{G}(\lambda,k),\quad
   \lambda\in(0,1],\ k\in\mathbb N.
\]
But this is equivalent to
\[
  (1+\hat{\delta}^{\alpha k})(\alpha - \lfloor \alpha \rfloor)
  <
  \Big(\alpha(1+\lambda)^{\alpha k}- \sum_{\omega \in K_\alpha}(1+\lambda \omega)^{\alpha k}\Big),
\]
which follows from the definition of $\hat{\delta}.$

\section{Proof of Theorem~\ref{thm:cnc} for $\lfloor\alpha\rfloor$ even}\label{se:even}

We now prove~\eqref{eq:cnc2} in the case that
\begin{equation}\label{eq:alpha2}
   2<\alpha \notin \mathbb N,\quad \lfloor \alpha \rfloor =2m,\quad m\in\mathbb N.
\end{equation}
As $\sin\alpha \pi>0,$ it follows from~\eqref{eq:binom} that~\eqref{eq:cnc2}
is equivalent to
\begin{multline}\label{eq:goal2}
  \int_0^1 F(t,\lambda,k)dt > \frac{\pi}{\alpha \lambda^\alpha \sin \alpha \pi}
    \Big(\sum_{\omega\in K_\alpha}(1+\lambda \omega)^{\alpha k}
    -\alpha(1+\lambda)^{\alpha k}\Big),\\
     \lambda \in (0,1],\ k\in\mathbb N.
\end{multline}
Since $F(\cdot,\lambda,k)$ is positive on $(0,1),$ the following lemma establishes this for $\lambda\in[\tfrac12,1].$
\begin{lemma}\label{le:big la}
  Let $\alpha>2$ be as in~\eqref{eq:alpha2}. Then we have
  \[
    \sum_{\omega\in K_\alpha}(1+\lambda \omega)^{\alpha k} \leq 
    \alpha(1+\lambda)^{\alpha k}, \quad \lambda \in [\tfrac12,1],\ k\in\mathbb N.
  \]
\end{lemma}
\begin{proof}
  By~\eqref{eq:K exp}, we have
  \begin{align}
    \sum_{\omega\in K_\alpha}(1+\lambda \omega)^{\alpha k}
    &\leq \sum_{\omega\in K_\alpha}|1+\lambda \omega|^{\alpha k}\notag \\
    &=\sum_{j=-m}^m \Big(1+2\lambda \cos
      \frac{2j\pi}{\alpha}+\lambda^2\Big)^{\alpha k/2},
      \quad \lambda \in(0,1],\ k\in\mathbb N. \label{eq:K abs}
  \end{align}
  We will show that
  \begin{equation}\label{eq:goal3a}
    \sum_{j=-m}^m \Big(1+2\lambda \cos
      \frac{2j\pi}{\alpha}+\lambda^2\Big)
      \leq \alpha(1+\lambda)^2,\quad \lambda \in[\tfrac12,1].
  \end{equation}
  Then, \eqref{eq:K abs} and~\eqref{eq:goal3a} imply
  \begin{align*}
      \sum_{j=-m}^m \Big(1+&2\lambda \cos
      \frac{2j\pi}{\alpha}+\lambda^2\Big)^{\alpha k/2}=
      (1+\lambda)^{\alpha k}
        \sum_{j=-m}^m \bigg(\frac{1+2\lambda \cos
        \frac{2j\pi}{\alpha}+\lambda^2}{(1+\lambda)^2}\bigg)^{\alpha k/2} \\
      &\qquad\qquad\leq (1+\lambda)^{\alpha k}
        \sum_{j=-m}^m \frac{1+2\lambda \cos
        \frac{2j\pi}{\alpha}+\lambda^2}{(1+\lambda)^2}
        \leq \alpha(1+\lambda)^{\alpha k},
  \end{align*}
  which proves the lemma.
  To prove~\eqref{eq:goal3a},
  observe that~\eqref{eq:alpha2} implies
  \begin{align*}
    |K_\alpha| = \Big \lfloor \frac{\alpha}{2} \Big \rfloor - 
        \Big \lfloor {-\frac{\alpha}{2}} \Big \rfloor 
       = m-(-m-1) = 2m+1 = \lceil \alpha \rceil.
  \end{align*}
  Using the geometric series to evaluate the cosine sum (see, e.g., p.~102
  in~\cite{Mi64}), we obtain
  \begin{multline}\label{eq:j sum}
     \sum_{j=-m}^m \Big(1+2\lambda \cos
      \frac{2j\pi}{\alpha}+\lambda^2\Big)
     = \lceil \alpha \rceil (1+\lambda^2) \\
     +2\lambda\bigg(
       1 + 2 \frac{\cos\big((m+1)\pi/\alpha\big)\sin(m\pi/\alpha)}
         {\sin(\pi/\alpha)}
     \bigg).
  \end{multline}
  Here, $\cos\big((m+1)\pi/\alpha\big)<0,$ and the sines are both positive.
  By~\eqref{eq:j sum} and the elementary inequalities
  \begin{align*}
    \cos x &\leq -1 +\tfrac12(x-\pi)^2,\quad x\in\mathbb R, \\
    \sin x &\leq x,\quad x\geq 0, \\
    \sin x &\geq x-\tfrac16 x^3,\quad x\geq 0,
  \end{align*}
  the following statement is sufficient for the validity of~\eqref{eq:goal3a}:
  \begin{multline}\label{eq:cad}
    (A+1)(1+\lambda^2)+2\lambda\bigg(
    1 + 2\frac{\big({-1}+\frac12\big(\frac{(M+1)\pi}{A}-\pi\big)^2\big)\frac{M\pi}{A}}
      {\frac{\pi}{A}-\frac16 (\frac{\pi}{A})^3}
      \bigg)
      \leq A(1+\lambda)^2,\\ \quad
      A>2,\ M>0,\ 2M < A < 2M+1,\ \tfrac12 \leq \lambda \leq 1. 
  \end{multline}
  This is a polynomial inequality in real variables
  with polynomial constraints, which can be verified
  by cylindrical algebraic decomposition, using a computer algebra system.
  For instance, using Mathematica's {\tt Reduce} command on~\eqref{eq:cad},
  with the first $\leq$ replaced by $>$, yields {\tt False}. This shows
  that~\eqref{eq:cad} is correct, which finishes the proof.
\end{proof}

The following two lemmas will be required for some estimates
of the function~$F(t,\lambda,k)$ for small~$\lambda.$
\begin{lemma}\label{le:int}
For $\alpha$ as in~\eqref{eq:alpha2}, we have
\begin{equation}\label{eq:int1}
  \int_0^\infty \frac{s^\alpha}{s^{2\alpha} -2s^\alpha \cos\alpha \pi+1}ds
  =\frac{\pi\sin\big(\frac{\lfloor \alpha \rfloor+1}{\alpha}\pi\big)}
        {\alpha \sin(\alpha \pi) \sin\big(\frac{\alpha+1}{\alpha}\pi\big)}
\end{equation}
and
\begin{equation}\label{eq:int2}
  \int_0^\infty \frac{s^{\alpha-1}}{s^{2\alpha} -2s^\alpha \cos\alpha \pi+1}ds
  =\frac{\pi (\lceil \alpha \rceil-\alpha)}{\alpha \sin \alpha \pi}.
\end{equation}
\end{lemma}
\begin{proof}
   By substituting $s^\alpha=w,$
  \begin{align*}
      \int_0^\infty \frac{s^\alpha}{s^{2\alpha} -2s^\alpha \cos\alpha \pi+1}ds
      &= \frac{1}{\alpha} \int_0^\infty \frac{w^{1/\alpha}}{w^{2} -2w\cos\alpha \pi+1}dw\\
&= \frac{1}{\alpha} \int_0^\infty \frac{w^{1/\alpha}}{w^{2} +2w\cos(\pi(\alpha-\lfloor \alpha \rfloor-1))+1}dw,
  \end{align*}
  where we have used that $\lfloor \alpha \rfloor$ is even.
  The first formula now follows from 3.242 on p.~322 of~\cite{GrRy07}, with $m=1/(2\alpha),$ $n=\tfrac12$ and
  $t=\pi(\alpha-\lfloor \alpha \rfloor-1).$ As for~\eqref{eq:int2},
   \begin{align*}
      \int_0^\infty \frac{s^{\alpha-1}}{s^{2\alpha} -2s^\alpha \cos\alpha \pi+1}ds
      &= \frac{1}{\alpha} \int_0^\infty \frac{1}{w^{2} -2w\cos\alpha \pi+1}dw\\
&= \frac{1}{\alpha} \int_0^\infty \frac{1}{w^{2} +2w\cos(\pi(\alpha-\lfloor \alpha \rfloor-1))+1}dw.
  \end{align*}
  The identity~\eqref{eq:int2} then follows from 11a) on p.~14
  of~\cite{GrHo73}, with
  \[
    \lambda = \pi(\alpha-\lfloor \alpha \rfloor-1)
    =\pi(\alpha-\lceil \alpha \rceil). \qedhere
  \]
\end{proof}
\begin{lemma}\label{le:int asympt}
  Again, suppose that $\alpha$ satisfies~\eqref{eq:alpha2}.
  For $k\in\mathbb N,$ we have
  \begin{multline*}
    \int_0^{1/\lambda} \frac{s^{\alpha-1}(1-\lambda s)^{\alpha k}}
      {s^{2\alpha}-2s^\alpha \cos\alpha \pi +1} ds=
     \int_0^{\infty} \frac{s^{\alpha-1}}
      {s^{2\alpha}-2s^\alpha \cos\alpha \pi +1} ds\\
     -\alpha k \lambda \int_0^{\infty} \frac{s^{\alpha}}
      {s^{2\alpha}-2s^\alpha \cos\alpha \pi +1} ds
      +\mathrm{o}(\lambda), \quad \lambda \downarrow 0.
  \end{multline*}
\end{lemma}
\begin{proof}
  Fix some $\beta\in(\tfrac{1}{\alpha},\tfrac12).$ 
  We have
  \begin{equation}\label{eq:a-1}
    \int_{\lambda^{-\beta}}^\infty \frac{s^{\alpha-1}}{s^{2\alpha} -2s^\alpha \cos\alpha \pi+1}ds
    =\mathrm{O}(\lambda^{\alpha \beta})
    =\mathrm{o}(\lambda),\quad \lambda\downarrow0,
  \end{equation}
  and
  \begin{equation}\label{eq:a}
    \int_{\lambda^{-\beta}}^\infty \frac{s^{\alpha}}{s^{2\alpha} -2s^\alpha \cos\alpha \pi+1}ds
    =\mathrm{O}(\lambda^{\beta(\alpha-1)})
    =\mathrm{o}(1),\quad \lambda\downarrow0.
  \end{equation}
  These assertions easily follow from the fact that the integrand is of
  order $\mathrm{O}(s^{-\alpha-1})$ resp.\ $\mathrm{O}(s^{-\alpha})$ at infinity.
  Moreover, we have the uniform expansion
  \begin{align}
    (1-\lambda s)^{\alpha k} &= 1 - \alpha k\lambda s + \mathrm{O}(\lambda^{2-2\beta}) \notag \\
    &=1 - \alpha k\lambda s +\mathrm{o}(\lambda),
    \quad 0\leq s\leq \lambda^{-\beta},\ \lambda\downarrow0. \label{eq:unif}
  \end{align}
  Since $0\leq(1-\lambda s)^{\alpha k}\leq 1,$ \eqref{eq:a-1} implies
  \begin{align*}
      \int_0^{1/\lambda} \frac{s^{\alpha-1}(1-\lambda s)^{\alpha k}}
      {s^{2\alpha}-2s^\alpha \cos\alpha \pi +1} ds
      &= \int_0^{\lambda^{-\beta}} \frac{s^{\alpha-1}(1-\lambda s)^{\alpha k}}
      {s^{2\alpha}-2s^\alpha \cos\alpha \pi +1} ds + \mathrm{o}(\lambda).
  \end{align*}
  The statement now follows from~\eqref{eq:unif}, \eqref{eq:a} and \eqref{eq:a-1}.
\end{proof}
We now continue the proof of~\eqref{eq:goal2}. 
{}From the definition of~$F$, it is clear that
\begin{align}
  \int_0^1 F(t,\lambda,k)dt &\geq 
  \int_0^1 \frac{t^{\alpha-1}(1-t)^{\alpha k}}{|t^\alpha -\lambda^\alpha e^{-i\alpha \pi}|^2} dt \notag\\
  &= \lambda^{-\alpha} \int_0^{1/\lambda} \frac{s^{\alpha-1}(1-\lambda s)^{\alpha k}}
    {s^{2\alpha}-2s^\alpha \cos\alpha \pi +1} ds. \label{eq:F lower}
\end{align}
As noted above, from~\eqref{eq:K exp}, we have
\begin{align}
  \sum_{\omega\in K_\alpha}(1+\lambda \omega)^{\alpha k}
  &\leq \sum_{\omega\in K_\alpha}|1+\lambda \omega|^{\alpha k} \notag \\
  &=1 + 2\sum_{j=1}^m \Big(1+2\lambda \cos
    \frac{2j\pi}{\alpha}+\lambda^2\Big)^{\alpha k/2},
    \quad \lambda \in(0,1],\ k\in\mathbb N. \label{eq:est sum}
\end{align}
Since Lemma~\ref{le:big la} settles the case
$\lambda\in[\tfrac12,1],$ we may assume $\lambda\in(0,\tfrac12)$ in what follows.
Using~\eqref{eq:F lower} and~\eqref{eq:est sum} in~\eqref{eq:goal2}, we see that
it is sufficient to show
\begin{multline}\label{eq:goal3}
  2\sum_{j=1}^m \Big(1+2\lambda \cos \frac{2j\pi}{\alpha}+\lambda^2\Big)^{\alpha k/2}
  +(1-\alpha)(1+\lambda)^{\alpha k}\\
  -\frac{\alpha \sin \alpha \pi}{\pi} \int_0^{1/\lambda} \frac{s^{\alpha-1}(1-\lambda s)^{\alpha k}}
    {s^{2\alpha}-2s^\alpha \cos\alpha \pi +1} ds < 0,
    \quad \lambda \in(0,\tfrac12),\ k\in\mathbb N.
\end{multline}
This will be proven in the following two lemmas.
\begin{lemma}\label{le:decrease}
 For $\lambda\in(0,\tfrac12),$ the left hand side of~\eqref{eq:goal3} decreases w.r.t.~$\lambda$.
\end{lemma} 
\begin{proof}
The derivative of this expression is
\begin{align}
  2\alpha k \sum_{j=1}^m&\Big( \cos \frac{2j\pi}{\alpha}+\lambda\Big)
     \Big(1+2\lambda \cos \frac{2j\pi}{\alpha}+\lambda^2\Big)^{\alpha k/2-1}
       + \alpha k(1-\alpha)(1+\lambda)^{\alpha k-1}\notag \\
     &\qquad +\frac{\alpha \sin \alpha \pi}{\pi} 
     \int_0^{1/\lambda} \frac{s^{\alpha}(1-\lambda s)^{\alpha k-1}}
    {s^{2\alpha}-2s^\alpha \cos\alpha \pi +1} ds \notag\\
   &\leq 2\alpha k \sum_{1\leq j\leq \alpha/3}\Big( \cos \frac{2j\pi}{\alpha}+\lambda\Big)
     \Big(1+2\lambda \cos \frac{2j\pi}{\alpha}+\lambda^2\Big)^{\alpha k/2-1}\notag\\
    &\qquad   + \alpha k(1-\alpha)(1+\lambda)^{\alpha k-1}
    +\frac{\alpha \sin \alpha \pi}{\pi} 
     \int_0^{\infty} \frac{s^{\alpha}}{s^{2\alpha}-2s^\alpha \cos\alpha \pi +1} ds. \label{eq:tbc}
\end{align}
Note that it is easy to show that $\cos(2j\pi/\alpha)+\lambda<0$ for $j>\alpha/3,$
which is where we use our assumption that $\lambda<\tfrac12$. Thus, we are discarding
only negative terms when passing from $\sum_{j=1}^m$ to $\sum_{1\leq j\leq \alpha/3}$
in~\eqref{eq:tbc}.
By Lemma~\ref{le:int},  the last term in~\eqref{eq:tbc} satisfies
\begin{equation}\label{eq:with sine}
  \frac{\alpha \sin \alpha \pi}{\pi} 
     \int_0^{\infty} \frac{s^{\alpha}}{s^{2\alpha}-2s^\alpha \cos\alpha \pi +1} ds
     = \frac{\sin\big(\frac{\lfloor \alpha \rfloor+1}{\alpha}\pi\big)}
        {\sin\big(\frac{\alpha+1}{\alpha}\pi\big)}
        < 1,
\end{equation}
where the inequality follows from
\begin{equation}\label{eq:floor alpha}
  1 < \frac{\lfloor \alpha \rfloor+1}{\alpha}
  < \frac{\alpha+1}{\alpha} < \frac32.
\end{equation}
We can thus estimate~\eqref{eq:tbc} further by
\begin{align*}
  2\alpha k\sum_{1\leq j\leq \alpha/3}
  (1+\lambda)&\big((1+\lambda)^2\big)^{\alpha k/2-1}
  +\alpha k(1-\alpha)(1+\lambda)^{\alpha k-1} +1 \\
  &= 2\alpha k \Big\lfloor \frac{\alpha}{3} \Big\rfloor(1+\lambda)^{\alpha k-1}
     +\alpha k(1-\alpha)(1+\lambda)^{\alpha k-1} +1\\
  &\leq\alpha k(1+\lambda)^{\alpha k-1}\Big(2\Big\lfloor \frac{\alpha}{3} \Big\rfloor+2-\alpha\Big)
  \leq 0.
\end{align*}
Indeed, it is easy to see that $2\lfloor \alpha/3 \rfloor+2-\alpha \leq0$
for $\alpha$ as in~\eqref{eq:alpha2}.
\end{proof}
Presumably, the preceding lemma can be extended to $\lambda\in(0,1],$
but  this would require a much better estimate for 
 $\sum_{\omega \in K_\alpha}(1+\lambda\omega)^{\alpha k}$
 than the one we have used.
The proof of~\eqref{eq:cnc2} could then possibly be streamlined, because Lemma~\ref{le:big la}
would no longer be required.
\begin{lemma}\label{le:la small}
  Let $\alpha$ be as in~\eqref{eq:alpha2}, and $k\in\mathbb N.$
  Then~\eqref{eq:goal3} holds for $\lambda>0$ sufficiently small.
\end{lemma}
\begin{proof}
Since $m=(\lceil \alpha \rceil-1) /2,$ expanding the first terms of~\eqref{eq:goal3} gives
\begin{multline*}
  \lceil \alpha \rceil -\alpha
  +\alpha k\bigg( 2\sum_{j=1}^m \cos \frac{2j\pi}{\alpha}+1-\alpha\bigg)\lambda
    +\mathrm{O}(\lambda^2)\\
  -\frac{\alpha \sin \alpha \pi}{\pi} \int_0^{1/\lambda} \frac{s^{\alpha-1}(1-\lambda s)^{\alpha k}}
    {s^{2\alpha}-2s^\alpha \cos\alpha \pi +1} ds.
\end{multline*}
By Lemmas~\ref{le:int} and~\ref{le:int asympt}, this further equals
\begin{align}
  \lceil \alpha \rceil -\alpha
  +\alpha& k\bigg( 2\sum_{j=1}^m \cos \frac{2j\pi}{\alpha}+1-\alpha\bigg)\lambda \notag \\
  &\qquad\qquad\quad+\alpha - \lceil \alpha \rceil +\alpha k \lambda 
    \frac{\sin\big(\frac{\lfloor \alpha \rfloor+1}{\alpha}\pi\big)}
        {\sin\big(\frac{\alpha+1}{\alpha}\pi\big)} +\mathrm{o}(\lambda)\notag \\
   &=\alpha k\Bigg( 2\sum_{j=1}^m \cos \frac{2j\pi}{\alpha}+1-\alpha
     +\frac{\sin\big(\frac{\lfloor \alpha \rfloor+1}{\alpha}\pi\big)}
        {\sin\big(\frac{\alpha+1}{\alpha}\pi\big)}
     \Bigg)\lambda +\mathrm{o}(\lambda), \label{eq:sharp}
\end{align}
and we see that~\eqref{eq:goal3} becomes sharp as $\lambda\downarrow0,$ as the left
hand side is $\mathrm{O}(\lambda).$ This is no surprise, since the
inequality~\eqref{eq:cnc2} we are proving is obviously sharp for $\lambda\downarrow0.$ 
It remains to show that the coefficient of~$\lambda$ in~\eqref{eq:sharp}
is negative.
Similarly to~\eqref{eq:tbc}, we have the bound
\begin{align*}
  2\sum_{j=1}^m \cos \frac{2j\pi}{\alpha}\leq 
    2\sum_{1\leq j\leq \alpha/3}\cos \frac{2j\pi}{\alpha}
    \leq 2\Big\lfloor \frac{\alpha}{3} \Big \rfloor.
\end{align*}
Thus, we wish to show that
\begin{equation}\label{eq:final goal}
   2\Big\lfloor \frac{\alpha}{3} \Big \rfloor +1-\alpha
   +\frac{\sin\big(\frac{\lfloor \alpha \rfloor+1}{\alpha}\pi\big)}
        {\sin\big(\frac{\alpha+1}{\alpha}\pi\big)} <0.
\end{equation}
As the sine quotient is $<1,$ by~\eqref{eq:with sine}, this is clearly true for $\alpha>6.$
Now consider $\alpha\in(4,5).$ It is easy to verify that
\[
  \sin\Big(\frac{4+1}{\alpha}\pi\Big)
  < -\frac{1}{\sqrt2}+ \frac{6\pi(\alpha-4)}{16\sqrt{2}},\quad 4<\alpha<5,
\]
and that
\[
  \sin\Big(\frac{\alpha+1}{\alpha}\pi\Big)
  > -\frac{1}{\sqrt2}+ \frac{\pi(\alpha-4)}{20\sqrt{2}},\quad 4<\alpha<5.
\]
Using these estimates in~\eqref{eq:final goal} leads to a quadratic inequality,
which is straightforward to check. The proof
of~\eqref{eq:final goal} for $\alpha\in(2,3)$ is similar.
\end{proof}
Clearly, Lemmas~\ref{le:decrease} and~\ref{le:la small} establish~\eqref{eq:goal3}.
As argued above~\eqref{eq:goal3}, we have thus proven~\eqref{eq:goal2},
hence~\eqref{eq:cnc2}, and the proof of Theorem~\ref{thm:cnc} is complete.

\bibliographystyle{siam}
\bibliography{literature}

\begin{thebibliography}{10}

\bibitem{BP55}
{\sc M.~Biernacki and J.~Krzy{\. z}}, {\em On the monotonicity of certain
  functionals in the theory of analytic functions}, Ann. Univ. Mariae
  Curie-Sk\l odowska A, 9 (1955), pp.~135--147.

\bibitem{BoGeSo20}
{\sc H.~Boedihardjo, X.~Geng, and N.~P. Souris}, {\em Path developments and
  tail asymptotics of signature for pure rough paths}, Adv. Math., 364 (2020),
  pp.~107043, 48.

\bibitem{Br64}
{\sc A.~M. Bruckner}, {\em Some relationships between locally superadditive
  functions and convex functions}, Proc. Amer. Math. Soc., 15 (1964),
  pp.~61--65.

\bibitem{Co78}
{\sc J.~B. Conway}, {\em Functions of one complex variable}, vol.~11 of
  Graduate Texts in Mathematics, Springer-Verlag, New York-Berlin, second~ed.,
  1978.

\bibitem{CrLiLy15}
{\sc D.~Crisan, C.~Litterer, and T.~J. Lyons}, {\em Kusuoka-{S}troock gradient
  bounds for the solution of the filtering equation}, J. Funct. Anal., 268
  (2015), pp.~1928--1971.

\bibitem{El16}
{\sc S.~K. Elagan}, {\em On the invalidity of semigroup property for the
  {M}ittag-{L}effler function with two parameters}, J. Egyptian Math. Soc., 24
  (2016), pp.~200--203.

\bibitem{FrRi13}
{\sc P.~Friz and S.~Riedel}, {\em Integrability of (non-)linear rough
  differential equations and integrals}, Stoch. Anal. Appl., 31 (2013),
  pp.~336--358.

\bibitem{FrRi14}
\leavevmode\vrule height 2pt depth -1.6pt width 23pt, {\em Convergence rates
  for the full {G}aussian rough paths}, Ann. Inst. Henri Poincar\'{e} Probab.
  Stat., 50 (2014), pp.~154--194.

\bibitem{GoKiMaRo20}
{\sc R.~Gorenflo, A.~A. Kilbas, F.~Mainardi, and S.~Rogosin}, {\em
  Mittag-{L}effler functions, related topics and applications}, Springer
  Monographs in Mathematics, Springer, Berlin, second~ed., 2020.

\bibitem{GrRy07}
{\sc I.~S. Gradshteyn and I.~M. Ryzhik}, {\em Table of integrals, series, and
  products}, Elsevier/Academic Press, Amsterdam, seventh~ed., 2007.

\bibitem{GrHo73}
{\sc W.~Gr\"{o}bner and N.~Hofreiter}, {\em Integraltafel. {Z}weiter {T}eil.
  {B}estimmte {I}ntegrale}, Springer-Verlag, fifth~ed., 1973.

\bibitem{HaHi10}
{\sc K.~Hara and M.~Hino}, {\em Fractional order {T}aylor's series and the
  neo-classical inequality}, Bull. Lond. Math. Soc., 42 (2010), pp.~467--477.

\bibitem{HaMaSa11}
{\sc H.~J. Haubold, A.~M. Mathai, and R.~K. Saxena}, {\em Mittag-{L}effler
  functions and their applications}, J. Appl. Math.,  (2011), pp.~Art. ID
  298628, 51.

\bibitem{In12}
{\sc Y.~Inahama}, {\em A moment estimate of the derivative process in rough
  path theory}, Proc. Amer. Math. Soc., 140 (2012), pp.~2183--2191.

\bibitem{Ly98}
{\sc T.~J. Lyons}, {\em Differential equations driven by rough signals}, Rev.
  Mat. Iberoamericana, 14 (1998), pp.~215--310.

\bibitem{Mi64}
{\sc D.~S. Mitrinovi{\'c}}, {\em Elementary Inequalities}, P. Noordhoff Ltd.,
  1964.

\bibitem{Os71}
{\sc T.~J. Osler}, {\em Taylor's series generalized for fractional derivatives
  and applications}, SIAM J. Math. Anal., 2 (1971), pp.~37--48.

\bibitem{PeLi10}
{\sc J.~Peng and K.~Li}, {\em A note on property of the {M}ittag-{L}effler
  function}, J. Math. Anal. Appl., 370 (2010), pp.~635--638.

\bibitem{Sato99}
{\sc K.-I. Sato}, {\em L\'evy {P}rocesses and {I}nfinitely {D}ivisible
  {D}istributions}, Cambridge Studies in Advanced Mathematics, Cambridge
  University Press, Cambridge, second~ed., 2013.

\bibitem{ScSoVo12}
{\sc R.~L. Schilling, R.~Song, and Z.~Vondra\v{c}ek}, {\em Bernstein
  {F}unctions: {T}heory and {A}pplications.}, vol.~37 of De Gruyter Studies in
  Mathematics, Walter de Gruyter \& Co., Berlin, second~ed., 2012.

\bibitem{TS15}
{\sc T.~Simon}, {\em Mittag-{L}effler functions and complete monotonicity},
  Integral Transforms Spec. Funct., 26 (2015), pp.~36--50.

\bibitem{Wi05}
{\sc A.~Wiman}, {\em \"{U}ber die {N}ullstellen der {F}unktionen {$E_a(x)$}},
  Acta Math., 29 (1905), pp.~217--234.

\end{thebibliography}

\end{document}